\documentclass[12pt]{amsart}

\usepackage[utf8]{inputenc}
\usepackage[T1]{fontenc}
\usepackage{geometry}

\usepackage{amsmath, amssymb}
\usepackage{appendix}

\newtheorem*{define}{Definition}
\newtheorem*{thm}{Theorem}
\newtheorem*{lem}{Lemma}
\newtheorem*{rem}{Remark}
\newtheorem{Define}{Definition}[section]
\newtheorem{Thm}[Define]{Theorem}
\newtheorem{Lem}[Define]{Lemma}
\newtheorem{Cor}[Define]{Corollary}
\newtheorem{Prop}[Define]{Proposition}

\DeclareMathOperator{\C}{\mathbb{C}}
\DeclareMathOperator{\R}{\mathbb{R}}

\DeclareMathOperator{\Z}{\mathbb{Z}}
\DeclareMathOperator{\N}{\mathbb{N}}
\DeclareMathOperator{\Hom}{Hom}
\DeclareMathOperator{\Aut}{Aut}
\DeclareMathOperator{\id}{Id}

\DeclareMathOperator{\SL}{SL}
\DeclareMathOperator{\Tr}{Tr}
\DeclareMathOperator{\Mod}{Mod}
\DeclareMathOperator{\ML}{ML}

\DeclareMathOperator{\Val}{\mathcal{V}}
\DeclareMathOperator{\Curv}{\mathcal{C}}

\title{Automorphisms of character varieties}

\begin{document}

\author{Julien March\'e}
\address{Sorbonne Universit\'e, IMJ-PRG, 75252 Paris c\'edex 05, France}
\email{julien.marche@imj-prg.fr}
\author{Christopher-Lloyd Simon} 
\address{ENS de Lyon, 15 parvis René Descartes, 69342 Lyon C\'edex 07, France}
\email{christopher-lloyd.simon@ens-lyon.fr}
\maketitle

\begin{abstract}
We show that the algebraic automorphism group of the SL$_2(\C)$ character variety of a closed orientable surface with negative Euler characteristic is a finite extension of its mapping class group. Along the way, we provide a simple characterization of the valuations on the character algebra coming from measured laminations.  
\end{abstract}

\section{Introduction}
Let $\Sigma$ be a closed oriented surface of genus $g\ge 2$ and denote $\Mod(\Sigma)$ its mapping class group which by the Dehn-Nielsen-Baer theorem can be identified with $\operatorname{Out}(\pi_1(\Sigma))=\Aut(\pi_1(\Sigma))/\operatorname{Inn}(\pi_1(\Sigma))$.

The meaningful spaces which carry an action of this group often manifest rigidity properties as explained for instance in \cite{Aramayona-souto_rigidity-map_2016} : whenever the mapping class group acts preserving some structure, it is almost the full automorphism group ; celebrated examples are the Teichmuller space with its Kähler structure, or the curve complex with its simplicial structure.

Here we show the same kind of result for the character variety, that is the space $X(\Sigma)$ defined as the algebraic quotient of $\Hom(\pi_1(\Sigma),\SL_2(\C))$ by the conjugation action of $\SL_2(\C)$. 
Recall that this affine variety is constructed from its algebra of functions 
\[
\C[X(\Sigma)]=\C[\Hom(\pi_1(\Sigma),\SL_2(\C))]^{\SL_2(\C)}.
\]

The automorphism group $\Aut(X(\Sigma))$ of the affine variety $X(\Sigma)$ is by definition the group of automorphisms of the $\C$-algebra $\C[X(\Sigma)]$.
The mapping class group acts algebraically on the character variety by setting $[\phi].[\rho]=[\rho\circ \phi^{-1}]$ for $\phi\in \Aut(\pi_1(\Sigma))$ and $\rho\in \Hom(\pi_1(\Sigma),\SL_2(\C))$.

On the other hand, any morphism $\rho:\pi_1(\Sigma)\to\SL_2(\C)$ can be multiplied by a central morphism $\lambda:\pi_1(\Sigma)\to \{\pm \id\}$. This provides an action of $H^1(\Sigma,\Z/2\Z)$ on $X(\Sigma)$. The purpose of this note is to prove the following:

\begin{thm}
When $g\ge 3$, the so defined map $H^1(\Sigma,\Z/2\Z)\rtimes\Mod(\Sigma)\to \Aut(X(\Sigma))$ is an isomorphism.
\end{thm}

Similar results were obtained in \cite{El-Huti} for the 1-punctured torus and 4-punctured sphere. We thank Serge Cantat for pointing this out and explaining the proof.

Our strategy is to consider the action of $\Aut(X(\Sigma))$ on a space $\mathcal{V}$ of valuations. Those are by definition the functions $v \colon \C[X(\Sigma)]\to \{-\infty\}\cup [0,+\infty)$ which are null on $\C^*$, take finite values except for $v(0)=-\infty$, and satisfy for all $f,g$ the relations $v(fg)=v(f)+v(g)$ and $v(f+g)\leq \max(v(f),v(g))$.
We endow $\mathcal{V}$ with the topology given by pointwize convergence, for which the action of $\Aut(X(\Sigma))$ is continuous.

Although it is established (for instance in \cite{Otal_compact-repres_2012}) that the space of measured laminations $\ML(\Sigma)$ embeds continuously in $\mathcal{V}$, it is not clear why this subset should be preserved by $\Aut(X(\Sigma))$ and this is one of the main steps in the proof.
This will follow from a clean and presumably new description of measured laminations as the set of ``simple valuations'' (Corollary \ref{simple=ML}) which is analogous to the set of monomial valuations for polynomial algebras. The definition is inspired from D. Thurston's formula concerning intersection of simple curves (see \cite{Dylan_Thurston_intersection_2009}).

While the computation of the automorphism group of the character variety had been asked on many occasions by Juan Souto to the first author ; this work stemmed from different motivations, including a better understanding of measured laminations in terms of valuations. We hope that this description will help understand the topology of the set of all valuations over the character variety, by showing for instance that it retracts by deformation on the space of simple valuations. 

Section 1 is quite elementary : we define simple valuations and relate them to measured laminations.
Section 2 is the technical heart of the note: using Morgan-Otal-Skora theorem, we show that any valuation is (sharply) dominated by a valuation associated to a measured lamination. Then we use a theorem of Masur stating that non uniquely ergodic measured lamination have measure $0$ to prove that $\Aut(X(\Sigma))$ acts on $\ML(\Sigma)$.
In Section 3, we explain how the geometric intersection of simple curves may be read on the valuation rings of their associated valuations, and deduce that $\Aut(X(\Sigma))$ acts on the curve complex. We conclude using Ivanov's Theorem.

\subsubsection*{Acknowledgements}
We wish to thank Javier Aramayona, Serge Cantat, Christopher Leininger, Marco Maculan, Mirko Mauri, Athanase Papadopoulos, Juan Souto and Alex Wright for their implication in this work.

\section{Simple valuations and measured laminations}

For $\alpha \in \pi_1(\Sigma)$, define $t_\alpha \in \C[X(\Sigma)]$ by the formula $t_\alpha([\rho])=\Tr \rho(\alpha)$. 
Let us recall a folklore presentation of the character variety, mainly due to Procesi \cite{ConciniProcesi_invariant-theory_2017}:

\begin{thm}
The algebra $\C[X(\Sigma)]$ is generated by the $t_\alpha$ for $\alpha\in \pi_1(\Sigma)$ with ideal of relations generated by $t_1-2$ and $t_\alpha t_\beta -t_{\alpha\beta}-t_{\alpha\beta^{-1}}$ for $\alpha,\beta\in \pi_1(\Sigma)$.
\end{thm}

We call \emph{multicurve} on $\Sigma$ an embedded one dimensional submanifold $\Gamma \subset \Sigma$ which is a union of curves homotopic to $\gamma_i \in\pi_1(\Sigma)\setminus \{1\}$, and set $t_\Gamma=\prod_{i=1}^n t_{\gamma_i}\in \C[X(\Sigma)]$. Components of $\Gamma$ must be simple and disjoint ; but we allow empty multicurves for which $t_\emptyset=1$.
The previous theorem has the following important consequence \cite{PrSi_sl2-skein_2000}:

\begin{thm}[Linear basis for the character algebra]
The family $(t_\Gamma)$ where $\Gamma$ ranges over the isotopy classes of multicurves forms a linear basis of $\C[X(\Sigma)]$.
\end{thm}

This allows to treat the elements $t_\Gamma$ as if they were monomials (despite the fact that they are not stable by multiplication), and motivates the next definition.

\begin{Define}[Simple valuations]
A valuation $v\in \mathcal{V}$ is called simple if for any $f\in \C[X(\Sigma)]\setminus\{0\}$ decomposed as $f=\sum m_\Gamma t_\Gamma$ in the basis of multicurves one has
\begin{equation}\label{simple}\tag{$\max$}
v(f)=\max\{v(t_\Gamma) \mid m_\Gamma \neq 0\}
\end{equation}
\end{Define}

\begin{Prop}[Measured laminations are simple valuations]
\label{laminationsimple}
Fix $\lambda \in \ML$ and denote by $i(\cdot,\cdot)$ the intersection pairing between measured laminations. The map $v_\lambda:t_\alpha\mapsto i(\lambda,\alpha)$ for $\alpha\in \pi_1(\Sigma)$ extends to a unique simple valuation $v_\lambda\in \mathcal{V}$.
\end{Prop}

\begin{proof}
We first treat the case of a measured lamination  associated to a simple curve $\delta$. We define $v_{\delta}$ on the basis of multicurves by the expression $v_{\delta}\left(\prod_{i=1}^n t_{\gamma_i}\right)= \sum i(\delta,\gamma_i)$ and extend it to $\C[X(\Sigma)]$ using equation \eqref{simple}.

Consider now any element $\alpha\in \pi_1(\Sigma)$: it can be decomposed as $t_\alpha=\sum m_\Gamma t_\Gamma$ in $\C[X(\Sigma)]$ and we must first prove that both definitions coincide ; meaning that $i(\delta, \alpha)=\max \{i(\delta,\Gamma) \mid m_\Gamma \neq 0\}$, but this is the content of D. Thurston's formula \cite[Lemma 12]{Dylan_Thurston_intersection_2009}.

All that remains is to check the formula $v_\delta(fg)=v_\delta(f)+v_\delta(g)$. For this, consider the increasing filtration of $\C[X(\Sigma)]$ defined by $F_n= \operatorname{Span} \{t_\gamma \mid \gamma\in \pi_1(\Sigma), i(\delta, \gamma)\leq n\}$. Let $k=v_\delta(f)$ and $l=v_\delta(g)$ be such that $f\in F_k\setminus F_{k-1}$ and $g\in F_l\setminus F_{l-1}$. It is equivalent to prove $v_\delta(fg)=v_\delta(f)+v_\delta(g)$ and that $fg$ is non zero in $F_{k+l}/F_{k+l-1}$. Hence, we are reduced to showing that the graded algebra $\bigoplus_{n\in \N}F_n/F_{n-1}$ is an integral domain, which is \cite[Theorem 12]{PrSi_skein_2019} (this can also be derived from the proof of \cite[Theorem 5.3]{CharleMar_independant-trace-su2_2012}).


For the general case, we extend the map $v_\lambda$ to $\C[X(\Sigma)]$ in the same way. The fact that it is indeed an extension and defines a simple valuation will follow from the case of simple curves by a limiting procedure.
We know from \cite[Theorème 1.3]{F-L-P_travaux-thurston_1979} that any measured lamination $\lambda$ is a limit of weighted simple curves $m_j \delta_j$ in the sense that for all $\alpha\in \pi_1(\Sigma)$ one has $i(\lambda, \alpha)=\lim_{j\to\infty} m_j i(\delta_j,\alpha)$.
Hence by equation \eqref{simple} the $m_jv_{\delta_j}$ converge pointwizely to $v_\lambda$, which indeed defines a valuation satisfying $v_\lambda(t_\alpha)=i(\lambda,\alpha)$ for all $\alpha\in\pi_1(\Sigma)$.
One could also observe that the subset of simple valuations is closed in $\mathcal{V}$ for the topology of pointwize convergence. 
\end{proof}

\begin{Define}[Strict valuations]
A valuation $v\in \Val$ is \emph{strict} if for all multicurves $\Gamma \neq \Delta$, we have $v(t_\Gamma)\neq v(t_\Delta)$.
\end{Define}

\begin{Define}[Positive valuations]
A valuation $v\in\mathcal{V}$ \emph{positive} when $v(t_\gamma)>0$ for every simple curve $\gamma$.
\end{Define}

\begin{rem} Comments on the relations between simple, strict and positive valuations:

\begin{itemize}

\item[(i)] For simple valuations it is equivalent to be positive on $t_\gamma$ for all simple $\gamma$, on $t_\alpha$ for all non-trivial $\alpha\in\pi_1(\Sigma)$ or on non-constant elements in $\C[X(\Sigma)]$.

\item[(ii)](Strict implies positive).
By definition $v(t_\emptyset)=v(1)=0$, it follows that if $v$ is strict then for any simple curve $\gamma$ we must have $v(t_\gamma)>0$. 

\item[(iii)](Strict implies simple).
For any valuation $v\in \Val$, the relation $v(f+g)\leq \max\{v(f),v(g)\}$ is an equality as soon as $v(f)\ne v(g)$. Therefore strict valuations are simple since they automatically satisfy the (\ref{simple}) relation.
\end{itemize}
\end{rem}

Recall that Thurston introduced a natural measure on the space of measured laminations.
Masur showed that up to scaling, there is only one $\Mod(\Sigma)$-invariant measure on $\ML$ in its Lebesgue class. Therefore it is proportional to the Borelian measure assigning to every open set $U$ the value
\[\lim_{R\to \infty} \frac{\operatorname{card} \{\textrm{multicurve }\Gamma\in R.U \}}{R^{6g-6}}.\]
This may serve as a definition for our purposes as we only speak about negligibility.

\begin{Prop}[Most measured laminations are strict]
\label{most_ml=strict}
The set of measured laminations $\lambda$ such that $v_\lambda$ is strict has full measure in $\ML$.
\end{Prop}

\begin{proof}
Let us first show that the set of $\lambda\in \ML$ such that $v_\lambda$ is positive has full measure.
By \cite[Section 3, Lemma 2]{Papa_remarques-ml_1986}, for every simple curve $\gamma$ the set $N(\gamma)=\{\lambda \in \ML \mid i(\gamma,\lambda)\neq 0\}$ is the complement of a codimension-$1$ PL-submanifold, in particular it has measure $0$ (see also \cite[Section 4]{Papa_intersection-hamiltonian_1986}). The set of positive laminations is the intersection of all $N(\gamma)$, so it has full measure.

The complement to the set of strict valuations is the union over the countable set of pairs of distinct non-empty multicurves, of the $N(\Gamma_1,\Gamma_2)=\{\lambda \in \ML \mid i(\lambda,\Gamma_1)= i(\lambda,\Gamma_2)\}$, so the result follows from the next lemma.
\end{proof}

\begin{Lem}[Intersection with distinct multicurves seldom coincides]
For distinct multicurves $\Gamma_1 \neq \Gamma_2$, the set $N(\Gamma_1,\Gamma_2)=\{\lambda \in \ML \mid i(\lambda,\Gamma_1)= i(\lambda,\Gamma_2)\}$ has measure $0$ in the set of positive measured laminations.
\end{Lem}

\begin{proof}
%
Consider $\lambda$ a positive measured lamination.
Since $\Gamma_1 \neq \Gamma_2$, there is a simple curve $\gamma$ such that $i(\gamma,\Gamma_1)\neq i(\gamma,\Gamma_2)$, also intersecting $\lambda$ by positivity.
Now apply to $\lambda$ the twist flow along $\gamma$ defined in \cite[Section 3]{Papa_remarques-ml_1986} (or \cite[Section 4]{Papa_intersection-hamiltonian_1986}) to get $\lambda_t = H^\gamma_t(\lambda)$. Then by \cite[Theorem 2]{Papa_remarques-ml_1986} (or \cite[Theorem 4]{Papa_intersection-hamiltonian_1986}), we have $i(\lambda_t, \Gamma_j) = i(\lambda, \Gamma_j) + t \cdot i(\gamma, \Gamma_j) + o(t)$ as $t$ decreases to $0$. This shows that the orbit of $\lambda$ intersects $N(\Gamma_1,\Gamma_2)$ at most once around $\lambda$.
%
%
This shows that the PL-submanifold $N(\gamma_1,\gamma_2)$ is transverse to $H^\gamma_t$, so it has positive codimension and measure $0$.
\end{proof}

\section{Domination of valuations}
\label{domination}

\begin{define}[Order structure]
We call \emph{domination} and denote $\le$ the partial order structure on $\Val$ defined by $v\leq w$ if for every $f\in\C[X(\Sigma)]$ we have $v(f)\leq w(f)$.
\end{define}

The main purpose of this subsection is to obtain the following theorem whose proof consists in a compilation of known results.

\begin{Thm}[Domination by measured laminations]
\label{domination_val_ml}
For any $v\in \mathcal{V}$ there exists a unique measured lamination $\lambda$ such that $v\leq v_\lambda$ and satisfying $v(t_\alpha)=v_\lambda(t_\alpha)$ for all $\alpha\in \pi_1(\Sigma)$.  
\end{Thm}

Denote $K=\C(X(\Sigma))$ the field of rational functions on the character variety. As usual, we extend $v$ to $K$ by setting $v(f/g)=v(f)-v(g)$. Notice however that our sign convention differs from the standard one in valuation theory, for instance our definition of the valuation ring associated to $v$ is $\mathcal{O}_v=\{f\in K, v(f)\le 0\}$.
We begin with the following standard lemma constructing the so-called tautological representation.

\begin{lem}[Tautological representation]
There exists a finite extension $\hat{K}$ of $K$ along with a representation $\rho:\pi_1(\Sigma)\to \SL_2(\hat{K})$ such that $\Tr\rho(\alpha)=t_\alpha$ for all $\alpha\in \pi_1(\Sigma)$.
\end{lem}
\begin{proof}
This follows from classical arguments in geometric invariant theory. In a nutshell, consider an algebraic closure $\overline{K}$ of $K$. One may interpret the inclusion $\C[X(\Sigma)]\to\overline{K}$ as a $\overline{K}$-point of $X(\Sigma)$.
As the map $\Hom(\pi_1(\Sigma),\SL_2(\overline{K}))\to X(\Sigma)(\overline{K})$ is surjective, we have a representation $\rho:\pi_1(\Sigma)\to \SL_2(\overline{K})$ satisfying $\Tr \rho(\gamma)=t_\gamma\in K$. As $\pi_1(\Sigma)$ is finitely generated, the coefficients of $\rho$ may be taken in a finite extension of $K$. One may also consult \cite[Proposition 3.3]{Marche_sl2-skein_2016} for a down-to-earth proof using K. Saito's theorem which even shows that a quadratic extension suffices.
\end{proof}

\begin{proof}[Proof of Theorem \ref{domination_val_ml}]

There exists (see \cite{Vaquie_extension-valuation_2007} for a proof) a valuation $\hat{v}:\hat{K}^*\to \R$ extending $v$, so we may consider the Bass-Serre real tree $T$ associated to the pair $(\hat{K},\hat{v})$. 
We know from \cite[Corollary III.1.2]{Morgan-Otal_1993} that there exists a measured lamination $\lambda$ with associated real tree $T_\lambda$ and an equivariant morphism of $\R$-trees $\Phi:T_\lambda\to T$ which decreases the distance (by Step 1 of Lemma I.1.1). Thus for any $\alpha \in \pi_1(\Sigma)$, the translation length of the action of $\alpha$ on $T_\lambda$, which equals $i(\lambda, \alpha)$, is greater than the translation length of $\alpha$ acting on $T$ which is $\max\{0,2\hat{v}(t_\alpha)\}=2v(t_\alpha)$. 

It follows in particular that $2v(t_\gamma)\le v_\lambda(t_\gamma)$ for every simple curve $\gamma$, therefore this holds also over multicurves. For any $f\in\C[X(\Sigma)]$ expanded as $f=\sum w_\Gamma t_\Gamma$, we get 
\[
v(f)\le \max\{v(t_\Gamma),w_\Gamma \neq 0\}\leq \frac{1}{2}\max\{v_\lambda(t_\Gamma),w_\Gamma\neq 0\}=\frac{1}{2}v_{\lambda}(f)
\]
where the last equality follows by Proposition \ref{laminationsimple}.

If we prove that $\Phi$ is an isometry on its image then we are done. Indeed this implies that the translation lengths of the actions of $\alpha\in \pi_1(\Sigma)$ on $T_\lambda$ and $T$ coincide, in other words that $v(t_\alpha)=\frac{1}{2}v_\lambda(t_\alpha)$.
But by Skora's theorem \cite{Skora_splittings_1996}, the morphism of $\R$-trees $\Phi:T_\lambda\to T$ mentionned above is an isometry on its image if and only if there does not exist any free subgroup $F_2\subset \pi_1(\Sigma)$ stabilizing a non-trivial edge in $\Phi(T_\lambda)$. 

Hence suppose by contradiction there exist $\alpha,\beta\in \pi_1(\Sigma)$ generating a free subgroup which fixes a non-trivial edge in $T$ of length $l$. Following for instance \cite[section 4.2]{Otal_compact-repres_2012}, this implies that up to conjugation, the tautological representation restricted to $F_2$ take its values in 
\[
G_l=\left\{\begin{pmatrix} a & b\\ c& d\end{pmatrix} \Big| \: ad-bc=1,\, v(a)\le 0,\, v(b)\le 0,\, v(c)\le -l,\, v(d)\le 0\right\}.
\]
If $\mathcal{M}^l_v$ denotes the ideal of $\mathcal{O}_v$ defined by the equation $v\le -l$ then $G_l$ consists precisely of those elements in $\SL_2(\mathcal{O}_v)$ projecting to triangular matrices in $\SL_2(\mathcal{O}_v/\mathcal{M}^l_v)$. The commutator $\rho([\alpha,\beta])$ is then unipotent in this quotient and we get $\Tr \rho([\alpha,\beta])=2 \mod{\mathcal{M}_v^l}$.
This means that $v(t_{[\alpha,\beta]}-2)\le -l$, contradicting the fact that $v$ is non-negative over $\C[X(\Sigma)]$.
\end{proof}


\begin{Cor}[Simple valuations are measured laminations]\label{simple=ML}
\label{simple=lamination}
A valuation $v\in\mathcal{V}$ is simple if and only if it has the form $v_\lambda$ for some measured lamination $\lambda$.
\end{Cor}


\begin{Cor}\label{aut_preserves_ml}
The action of $\Aut(X(\Sigma))$ on $\Val$ preserves $\ML$.
\end{Cor}

\begin{proof}
The main idea of the proof is to notice that the set $\mathcal{U}$ of \emph{untameable} valuations defined as those $u\in \Val$ such that for all $v\in \Val$, if $u\leq v$ then $v=Cu$ for some $C\geq 1$ is preserved by the action of $\Aut(X(\Sigma))$. We deduce from Theorem \ref{domination_val_ml} that $\mathcal{U}\subset \ML$ and we will show that $\mathcal{U}$ is dense in $\ML$. The next lemma provides sufficient conditions to be untameable. 
\begin{Lem}
If $\lambda \in \ML$ is strict and uniquely ergodic then $v_\lambda\in \Val$ is untameable.
\end{Lem}

\begin{proof}
Let $\lambda$ be a strict and uniquely ergodic measured lamination. 
Suppose there exists another valuation $w\in \Val$ with $v_\lambda \le w$. Applying Theorem \ref{domination_val_ml}, there exists another measured lamination $\mu$ with $w\le v_\mu$ and $w(t_\alpha)=v_\mu(t_\alpha)$ for all $\alpha \in \pi_1(\Sigma)$. So $v_\lambda \le v_\mu$ but as $\lambda$ is positive and uniquely ergodic, there exists $C\ge 1$ such that $\mu=C\lambda$. But $w$ and $C\lambda$ are strict and coincide on simple curves, so $w=C\lambda$. 
\end{proof}

We deduce from Masur's Theorem saying that the set of uniquely ergodic measured laminations has full measure in $\ML$ (see \cite[Theorem 2]{Masur_interval-exchange_1982}), that $\mathcal{U}$ is dense in $\ML$.
But $\ML$ is closed in $\Val$ as it coincides with the set of simple valuations by Corollary \ref{simple=lamination}, and $\Aut(X(\Sigma))$ acts continuously on $\Val$ so it preserves the closure of $\mathcal{U}$.
\end{proof}

\section{Discrete valuations and disjointness}

We call \emph{discrete} a valuation in $\mathcal{V}$ whose finite values belong to $\N$, and likewize a measured lamination whose associated valuation is discrete. One expects the underlying lamination to be a multicurve, this is almost true :

\begin{lem}
Discrete measured lamination are precisely the weighted multicurves $\frac{1}{2}\Gamma$ such that the class of $\Gamma$ in $H_1(\Sigma,\Z/2\Z)$ vanishes.
\end{lem}

\begin{proof}
This lemma can be shown using Dehn's coordinates, see \cite[Exposé 6]{F-L-P_travaux-thurston_1979}.
To give a proof more in the spirit of this note, we observe that the Bass-Serre tree associated to a discrete valuation has a simplicial structure. Let us mimick the proof of Theorem \ref{domination_val_ml}.

We extend the discrete valuation $v$ on $K$ to a discrete valuation $\hat{v}$ on $\hat{K}$ whose Bass-Serre tree $T$ is thus simplicial. The Morgan-Otal-Skora theorem reduces to a theorem of Stallings stating that the action of $\pi_1(\Sigma)$ on $T$ is dual to a multicurve $\Gamma$. This implies $v(t_\alpha)=\frac{1}{2}i(\Gamma,\alpha)$, so $v$ is associated to the measured lamination $\lambda=\frac{1}{2}\Gamma$.

Computing modulo $\Z$ the weighted intersection of $\frac{1}{2}\Gamma$ with any curve $\gamma$ one should get $0$: by Poincaré duality this shows that the class of $\Gamma$ in $H_1(\Sigma,\Z/2\Z)$ vanishes.
%
%
%
\end{proof}

An automorphism of $\C[X(\Sigma)]$ obviously preserves discrete valuations, hence by Corollary \ref{aut_preserves_ml} it preserves discrete measured laminations. It also preserves {\it indecomposable} discrete measured laminations, i.e. those that cannot be written as sum with positive coefficients of other ones. The previous lemma shows that indecomposable discrete measured laminations are precisely (possibly halves of) simple curves.

\begin{Prop}\label{preserve-curves}
The group $\Aut(X(\Sigma))$ preserve the set 
$$\Curv=\{v_\gamma |\gamma \textrm{ simple and non-separating}\}\cup \{\frac{1}{2}v_\gamma |\gamma \textrm{ simple and separating}\}.$$
\end{Prop}

To complete the proof of the main theorem our strategy is to apply Ivanov's theorem \cite{Ivanov_curve-complex_1997}, stating that the permutation group of $\Curv$ mapping disjoint curves to disjoint curves is reduced to the mapping class group $\Mod(\Sigma)$.
The key idea in showing preservation of disjointness is to consider for a simple curve $\gamma$ the ring $\mathcal{O}_\gamma^+=\C[X(\Sigma)]\cap \mathcal{O}_v$ where $v$ is the discrete valuation associated to $\gamma$. This is the coordinate ring for the image of the natural map $X(\Sigma)\to X(\Sigma\setminus\gamma)$: generically its fibers are orbits under the Goldman twist flow ; in particular we have $\dim \mathcal{O}_\gamma^+=\dim \C[X(\Sigma)] -1=6g-7$ where $\dim$ is the Krull dimension.

\begin{Lem}\label{preserve-disjointness}
For simple curves $\gamma$ and $\delta$ one has $\dim \mathcal{O}_\gamma^+\cap \mathcal{O}_\delta^+\le\dim \C[X(\Sigma)]-2$ with equality if and only if $i(\gamma,\delta)=0$.
\end{Lem}
\begin{proof}
Suppose that $\gamma$ and $\delta$ intersect minimally and set $\Sigma'=\Sigma\setminus (\gamma\cup\delta)$. The ring $\mathcal{O}_\gamma^+\cap\mathcal{O}_\delta^+$ is the image of the natural map $\C[X(\Sigma')]\to\C[\Sigma]$. Geometrically, it is the coordinate ring for the image of the restriction map $X(\Sigma)\to X(\Sigma')$.

Consider first the case when $\gamma$ and $\delta$ are disjoint. Let $U\subset X(\Sigma)$ be the Zariski open set containing characters of irreducible representations $\rho$ satisfying $\Tr \rho(\gamma)\ne \pm 2 \ne \Tr \rho(\delta)$. The natural map $X(\Sigma)\to X(\Sigma')$ restricted to $U$ and corestricted to its image is a fibration whose fibers are the orbits of the (commuting) Goldman twist flows along $\gamma$ and $\delta$.
We deduce that the dimension of its image which is the dimension of $\mathcal{O}_\gamma^+\cap \mathcal{O}_\delta^+$ is equal to $6g-6-2$.

Now suppose $\gamma$ and $\delta$ intersect: it is sufficient to show $\dim X(\Sigma')<\dim X(\Sigma) - 2$. 
The surface $\Sigma'$ is a disjoint union of connected surfaces with connected boundary: their fundamental group is either trivial or free of rank $\ge 2$. Denoting $F$ be the number of simply connected components, we have $\dim X(\Sigma')=-3\chi(\Sigma')+3F$.

%
%


%
But $\chi(\Sigma')=\chi(\Sigma)+i(\gamma,\delta)$ so $\dim X(\Sigma') = \dim X(\Sigma)-3i(\gamma,\delta)+3F$, and we must show that $i(\gamma,\delta)>F$.
Since $\gamma \cup \delta$ is a taut union of simple curves, the polygonal components have at least $4$ corners, so $4F$ is smaller or equal to the total number of corners which is $4i(\gamma,\delta)$.
But $F=i(\gamma,\delta)$ implies that $\Sigma'$ is a union of quadrilaterals, so $F=\chi(\Sigma')=\chi(\Sigma)+i(\gamma,\delta)$ and thus $\chi(\Sigma)=0$ : a contradiction.
\end{proof}
Set $\Mod'(\Sigma)=\Mod(\Sigma)$ if $g\ge 3$ and $\Mod'(\Sigma)=\Mod(\Sigma)/\langle \tau\rangle$ if $g=2$ where $\tau$ denotes the hyperelliptic involution. 
\begin{Thm}
There is a natural split extension 
\[
0\to H^1(\Sigma,\Z/2\Z)\to \Aut(X(\Sigma))\to \Mod'(\Sigma)\to 0.
\]
\end{Thm}

\begin{proof}
Let $\phi$ be an automorphism of $X(\Sigma)$. By Proposition \ref{preserve-curves} and Lemma \ref{preserve-disjointness}, $\phi$ acts on $\mathcal{C}$ by preserving the disjointness relation. Ivanov's Theorem (see \cite{Ivanov_curve-complex_1997} or \cite{Luo_complex-curves_2000}) implies that there is an element $g\in\Mod(\Sigma)$ which acts in the same way: this defines the map $\Aut(X(\Sigma))\to\Mod(\Sigma)$.
As $\Mod(\Sigma)$ naturally acts on $X(\Sigma)$ by sending $t_\gamma$ to $t_{\phi(\gamma)}$, there is a natural section of the previous map. Its kernel consists in mapping classes which fix all non-oriented simple curves. This reduces to the identity or the hyperelliptic involution in genus 2 (see for instance \cite{FaMar_primer-map_2012}) and gives the surjection $\Aut(X(\Sigma))\to \Mod'(\Sigma)$.

Consider an element $\phi\in\Aut(X(\Sigma))$ which acts trivially on $\mathcal{C}$. Fix a simple curve $\gamma$ and consider the possible values of $\phi(t_\gamma)$. As $\phi$ preserves the valuations $v_\delta$ for any simple curve $\delta$ we get $v_\delta(\phi(t_\gamma))=v_\delta(t_\gamma)=i(\delta,\gamma)$. 
Write $\phi(t_\gamma)=\sum m_\Gamma t_\Gamma$ : for any simple curve $\delta$ which does not intersect $\gamma$ one has $0=v_\delta(\phi(t_\gamma))=\max\{i(\delta,\Gamma)\mid m_\gamma\ne 0\}$, so every $\Gamma$ such that $m_\Gamma \ne 0$ must be a family of parallel copies of $\gamma$. Hence $\phi(t_\gamma)$ is a polynomial in $t_\gamma$ and as $\phi$ induces an automorphism of the subalgebra $\C[t_\gamma]$, one has $\phi(t_\gamma)=a_\gamma t_\gamma+b_\gamma$ with $a_\gamma\ne 0$.

Concentrate first on the subset of non-separating simple curves. For such elements $\gamma$ and $\delta$ intersecting once, applying $\phi$ to the relation $t_\gamma t_\delta=t_{\gamma\delta}+t_{\gamma\delta^{-1}}$ we get $b_\gamma=b_\delta=0$ and $a_\gamma a_\delta=a_{\gamma\delta}=a_{\gamma\delta^{-1}}$. The curves $\gamma\delta$ and $\gamma\delta^{-1}$ are homologous modulo 2 and the previous relation implies that $a$ factors as a map $a:H_1(\Sigma,\Z/2\Z)\to \C^*$. The equation $a_\gamma a_\delta=a_{\gamma\delta}$ implies that $a$ is a morphism, that is an element of $H^1(\Sigma,\Z/2\Z)$. 

Since this group acts on $\C[X(\Sigma)]$ by sending $t_\gamma$ to $a_\gamma t_\gamma$, we are reduced to the case where all $t_\gamma$ for $\gamma$ non-separating are fixed. A simple exercise in trace relations shows that these elements generate $\C[X(\Sigma)]$ as an algebra, therefore $\phi=a$ as desired.
\end{proof}

\bibliographystyle{alpha}
\bibliography{acv.bib}

\end{document}